\documentclass[draft,12pt]{article}

\usepackage{amsfonts,amsmath,amsthm,amscd,amssymb,latexsym,cite,verbatim,texdraw,floatflt,caption2,pb-diagram}

\usepackage[mathscr]{eucal}

\newtheorem{theorem}{Theorem}[section]
\newtheorem{lemma}{Lemma}[section]

\newtheorem{corollary}{Corollary}[section]
\theoremstyle{definition}

\newcommand{\keywords}{\textbf{Keywords: } }

\newcommand{\subjclass}{\textbf{Mathematics Subject Classification (2010):} }
\renewcommand{\abstract}{\textbf{Abstract.} }

\numberwithin{equation}{section}

\begin{document}


\title{Direct and inverse approximation theorems  \\
 in the  Besicovitch-Museilak-Orlicz spaces \\of almost periodic functions }

\author{  Stanislav Chaichenko, Andrii Shidlich, Tetiana Shulyk}



\date{}

\maketitle


\abstract{In  terms of the best approximations of functions and  generalized moduli of smoothness, direct and inverse approximation  theorems are proved  for  Besicovitch almost periodic functions whose Fourier exponent sequences have a single limit point in infinity and their Orlicz norms are finite.  Special attention is paid to the study of cases when the constants in these theorems are unimprovable.}





\section{Introduction}

The establishment of connections between the difference and differential properties of the function being approximated and the value of the error of its approximation by some methods was originated in the well-known works of Jackson (1911) and Bernstein (1912), in which the first direct and inverse approximation theorems were obtained. Subsequently, similar studies were carried out by many authors for various functional classes and for various approximating aggregates, and their results constitute the classics of modern approximation theory. Moreover, the exact results (in particular, in the sense of unimprovable constants) deserve special attention. A fairly complete description of the results on obtaining direct and inverse approximation theorems is contained in the monographs \cite{Butzer_Nessel_M1971, A_Timan_M1960, Stepanets_M2005, M_Timan_M2009}, etc.

In spaces of almost periodic functions, direct approximation theorems were established in the papers
\cite{Bredikhina_1968, Prytula_1972, Prytula_Yatsymirskyi_1983, Babenko_Savela_2012, Serdyuk-Shidlich_2022}, etc.
 In particular, Prytula \cite{Prytula_1972} obtained direct approximation theorem  for  Besicovitch almost periodic functions
 of the order $2$ ($B_2$-a.p. functions) in terms of the best approximations of functions and  their moduli of continuity.
 In \cite{Prytula_Yatsymirskyi_1983} and \cite{Babenko_Savela_2012},  such theorems were obtained, respectively, with moduli
  of smoothness  of $B_2$-a.p. functions of arbitrary positive integer order and with generalized moduli of smoothness. In
  \cite{Serdyuk-Shidlich_2022},  direct and inverse approximation theorems  were obtained in the Besicovitch-Stepanets spaces
   $B{\mathcal S}^{p}$. The main goal of this article is to obtain such theorems in  the Besicovitch-Museilak-Orlicz spaces
   $B{\mathcal S}_{\bf M}$.  These spaces are natural generalizations of the all  spaces mentioned above, and the results
   obtained can be viewed as an extension of these results 
   to the spaces $B{\mathcal S}_{\bf M}$.

\section{Preliminaries}

\subsection{Definition of the spaces $B{\mathcal S}_{\bf M}$}

Let  $ {B}^s$,  $ 1\le s<\infty$, be the space of all functions  Lebesgue summable with the $s$th degrees in each finite interval of the real axis, in which the distance is defined by the equality
\[
    D_{_{\scriptstyle  B^s}}(f,g)=\Big(\mathop{\overline{\lim}}
    \limits_{T\to \infty}\frac 1{2T}\int_{-T}^T  |f(x)-g(x)|^s {\mathrm d}x\Big)^{1/s}.
\]
Further, let ${\mathscr T}$ be the set of all trigonometric sums of the form  $\tau_N(x)=\sum_{k=1}^N a_k {\mathrm e}^{{\mathrm i} \lambda_kx}$, $N\in {\mathbb N}$, where $\lambda_k$ and $a_k$ are arbitrary real and complex numbers ($\lambda_k\in {\mathbb R}$, $a_k\in {\mathbb C}$).

An arbitrary function $f$ is called a Besicovitch almost periodic function of order $s$ (or $B^s$-a.p. function) and is denoted by $f\in B^s$-a.p.  \cite[Ch.~5, \S10]{Levitan_M1953}, \cite[Ch.~2, \S7]{Besicovitch_M1955}, if there exists a sequence of trigonometric sums
$\tau_1, \tau_2, \ldots$ from the set  ${\mathscr T}$ such that
\[
    \lim_{N\to \infty} D_{_{\scriptstyle  B^s}}(f,\tau_N)=0.
\]

If $s_1\ge s_2\ge 1$, then (see, for example, \cite{Bredikhina_1968,  Bredikhina_1984}) $B^{s_1}$-a.p.$\subset B^{s_2}$-a.p.$\subset B$-a.p.,
where  $B$-a.p.$:=B^1$-a.p.   For any $B$-a.p. function $f$, there exists the average value
\[
    A\{f\}:=\lim\limits_{T\to \infty} \frac {1}{ T} \int_0^T f(x){\mathrm d}x.
\]
The value of the function $A\{f(\cdot) {\mathrm e}^{-{\mathrm i}\lambda \cdot} \}$, $\lambda\in {\mathbb R}$, can be nonzero at most on a countable set.  As a result of numbering the values of this set in an arbitrary order, we obtain a set ${\mathscr S}(f)=\{\lambda_k\}_{k\in {\mathbb N}}$  of Fourier exponents, which is called the spectrum of the function $f$.
The numbers  $A_{\lambda_k}=A_{\lambda_k}(f)=A\{f(\cdot) {\mathrm e}^{-{\mathrm i}\lambda_k \cdot} \}$ are called the Fourier coefficients of the function  $f$. To each function $f\in B$-a.p. with spectrum ${\mathscr S}(f)$  there corresponds a
Fourier series of the form
 $
 \sum_k A_{\lambda_k} {\mathrm e}^{ {\mathrm i}\lambda_k x}.
$
If, in addition, $f\in B_2$-a.p., then the Parseval equality holds (see, for example, \cite[Ch.~2, \S9]{Besicovitch_M1955})
\[
    A\{|f|^2\}=\sum_{k\in {\mathbb N}} |A_{\lambda_k}|^2.
\]

Further, we will consider only those almost periodic functions from the spaces $B{\mathcal S}^p$, the sequences of Fourier exponents of which have a single limit point at infinity. For such functions $f$,  the Fourier series are written in
the symmetric form:
\begin{equation}\label{Fourier_Series}
    S[f](x)=\sum _{k \in {\mathbb Z}}
    A_{k} {\mathrm e}^{ {\mathrm i}\lambda_k x},\quad  \mbox{\rm  where }\
    A_{k}=A_{k}(f)=A\{f(\cdot) {\mathrm e}^{-{\mathrm i}\lambda_k \cdot}\},
\end{equation}
$\lambda_0:=0$, $\lambda_{-k}=-\lambda_k$,  $|A_k|+|A_{-k}|>0$, $\lambda_{k+1}>\lambda_k>0$ for $k>0$.


Let ${\bf M}=\{M_k(t)\}_{k\in {\mathbb Z}}$, $t\ge 0$, be a sequence of Orlicz functions. In other words, for every $k\in {\mathbb Z}$, the
function $M_k(t)$ is a nondecreasing convex function for which $M_k(0)=0$ and $M_k(t)\to \infty$ as $t\to \infty$.
Let
 ${\bf {M^*}}=\{{M^*_k}(v)\}_{k\in {\mathbb Z}}$ be the sequence of  functions defined by the relations
 \[
      {M^*_k}(v):=\sup\{uv-M_k(u): ~u\ge 0\}, \quad k \in \mathbb{Z}.
 \]
Consider the set $\Gamma=\Gamma({\bf {M^*}})$ of
sequences of positive numbers $\gamma=\{\gamma_k\}_{k\in \mathbb{Z}}$
such that  $\sum_{k\in \mathbb{Z}}{M^*_k}(\gamma_k){\le} 1$.
The modular space (or  Musilak-Orlicz  space) $B{\mathcal S}_{\bf M}$  is the space of all functions $f$ ($f\in B$-a.p.)   such that
the following quantity (which is also called the Orlicz norm of $f$) is finite:
\begin{equation} \label{def-Orlicz-norm}
    \|f\|_{_{\scriptstyle  {\bf M}}}:=    \| \{A_k \}_{k\in {\mathbb Z}}
    \|_{_{\scriptstyle l_{\bf M}({\mathbb Z})}}
    := \sup \Big\{ \sum\limits_{k \in \mathbb{Z}}
    \gamma_k|A_k(f) |: \quad  \gamma\in \Gamma({\bf {M^*}})\Big\}.
\end{equation}
By definition, $B$-a.p. functions are considered identical in $B{\mathcal S}_{\bf M}$ if they have the same Fourier series.

The spaces $B{\mathcal S}_{\bf M}$ defined in this way are Banach spaces.
Functional spaces of this type  have been studied by mathematicians since the 1940s (see, for example, the monographs \cite{Lindenstrauss-1977}, \cite{Musielak-1983}, \cite{Rao_Ren_2002}). In particular, the subspaces ${\mathcal S}_{\bf M}$ of all $2\pi$-periodic functions from $B{\mathcal S}_{\bf M}$ were considered in
\cite{Abdullayev_Chaichenko_Shidlich_2021, Abdullayev_Chaichenko_Shidlich_2021_RMJ}. If all the functions $M_k$ are identical (namely, $M_k(t)\equiv M(t)$,  $k\in {\mathbb Z}$), the spaces ${\mathcal S}_{\bf M}$  coincide with the ordinary Orlicz type spaces ${\mathcal S}_{M}$  \cite{Chaichenko_Shidlich_Abdullayev_2019}.
If $M_k(t)=\mu_k t^{p_k}$,  $p_k\ge 1$, $\mu_k\ge 0 $, then  ${\mathcal S}_{\bf M}$  coincide with the weighted spaces
${\mathcal S}_{_{\scriptstyle  \mathbf p,\,\mu}}$ with variable exponents \cite{Abdullayev_Chaichenko_Imash_kyzy_Shidlich_2020}.

 If all  functions $M_k(u)=u^p\Big(p^{-1/p}q^{-1/p'}\Big)^p$, $p>1$, $1/p+1/p'=1$, then  $B{\mathcal S}_{\bf M}$ are the Besicovitch-Stepanets spaces $B{\mathcal S}^{p}$ \cite{Serdyuk-Shidlich_2022} of functions  $f\in B$-a.p. with the norm
 \begin{equation}\label{norm_Sp}
 \|f\|_{_{\scriptstyle  {\bf M}}}=
 \|f\|_{_{\scriptstyle  B{\mathcal S}^p}} =\|\{A_{\lambda_k}(f)\}_{k\in {\mathbb N}}\|_{_{\scriptstyle  l_p({\mathbb N})}} =
 \Big(\sum_{k\in {\mathbb N}}|A_{\lambda_k}(f)|^p\Big)^{1/p}.
\end{equation}

The subspaces of all $2\pi$-periodic Lebesgue summable functions from  $B{\mathcal S}^{p}$  coincide with the well-known spaces  ${\mathcal S}^{p}$ (see, for example, \cite[Ch. XI]{Stepanets_M2005}).  For $p=2$, the sets $B{\mathcal S}^p=B{\mathcal S}^2$ coincide with the sets of $B_2$-a.p. functions and the spaces  ${\mathcal S}^{p}$  with the ordinary Lebesgue spases
of $2\pi$-periodic square-summable
functions, i.e., ${\mathcal S}^{2}=L_2.$


By $G_{\lambda_n}$  we denote the set of all $B$-a.p. functions whose Fourier exponents belong to the interval
$(-\lambda_n,\lambda_n)$  and define the value of the best approximation of $f\in B{\mathcal S}_{\bf M}$ by the equality
\begin{equation}\label{Best_Approximation_Deff}
    E_{\lambda_n}(f)_{_{\scriptstyle p}} =
    E_{\lambda_n}(f)_{_{\scriptstyle  B{\mathcal S}^p}}  =
    \inf\limits_{g\in G_{\lambda_n}}  \|f-g \|_{_{\scriptstyle p}}.
\end{equation}


\subsection{Generalized moduli of smoothness}


Let $\Phi$ be the set  of all continuous bounded nonnegative pair functions
$\varphi(t)$ such that $\varphi(0)=0$ and the Lebesgue measure of the set $\{t\in {\mathbb R}:\,\varphi(t)=0\}$
is equal to zero. For an arbitrary fixed $\varphi\in \Phi$, consider the generalized modulus of smoothness of a
function $f \in B{\mathcal S}_{\bf M}$
\begin{equation}\label{general_modulus}
    \omega_\varphi(f,\delta)_{_{\scriptstyle  {\bf M}}}:= \sup\limits_{|h|\le \delta} \sup \Big\{ \sum\limits_{k \in \mathbb{Z}}
    \gamma_k \varphi (\lambda_k h)|A_k(f) |:~  \gamma\in \Gamma \Big\},  \quad  \delta\ge 0.
\end{equation}

Consider the connection between the modulus (\ref{general_modulus}) and some well-known moduli of smoothness.
Let $\Theta=\{\theta_j\}_{j=0}^m$ be a nonzero  collection of complex numbers such that $\sum_{j=0}^m \theta_j=0$. We
associate the collection $\Theta$ with the difference operator
$
    \Delta_h^{\Theta}(f)=\Delta_h^{\Theta}(f,t)=\sum_{j=0}^m \theta_j f(t-jh)
$
and the modulus of smoothness
 \[
    \omega_{\Theta}(f,\delta)_{_{\scriptstyle {\bf M}}} :=
    \sup\limits_{|h|\le \delta}\|\Delta_h^{\Theta} (f)\|_{_{\scriptstyle {\bf M}}}.
 \]
Note that the collection $\Theta(m)=\Big\{\theta_j=(-1)^j {m \choose j}, \ j=0,1,\ldots, m\Big\}$, $m\in {\mathbb N}$, corresponds to the classical modulus of smoothness of order  $m$, i.e.,
$$
    \omega_{\Theta(m)}(f,\delta)_{_{\scriptstyle {\bf M}}}=
    \omega_m(f,\delta)_{_{\scriptstyle {\bf M}}}.
$$

For any  ${k}\in {\mathbb Z}$, the Fourier coefficients of the function $\Delta_h^\Theta (f)$ satisfy the equality
\[
    |A_k(\Delta_h^{\Theta}(f))|=
    |A_k(f)| \Big|\sum_{j=0}^m \theta_j {\mathrm e}^{-{\mathrm i}\lambda_k jh}\Big|.
\]
Therefore, taking into account (\ref{def-Orlicz-norm}), we see that for $\varphi_{\Theta}(t)= |\sum_{j=0}^m \theta_j {\mathrm e}^{-{\mathrm i} jt} |$,
 $
 \omega_{\varphi_{\Theta}}(f,\delta)_{_{\scriptstyle {\bf M}}} =
 \omega_{\Theta}(f,\delta)_{_{\scriptstyle {\bf M}}}.
 $
In particular, for  $\varphi_m(t)= 2^m   |\sin (t/2) |^m = 2^{\frac { m   }2} (1-\cos t)^{\frac m2}$, $m\in {\mathbb N}$, we have
$
 \omega_{\varphi_m}(f,\delta)_{_{\scriptstyle {\bf M}}} =
 \omega_m(f,\delta)_{_{\scriptstyle {\bf M}}}.
 $

Further, let
$$
    F_h (f,t)=f_h(x)
    :=\frac{1}{2h} \int\limits_{t-h}^{t+h} f(u) du
$$
be the Steklov function of a function $f \in B{\mathcal S}_{\bf M}.$ Define the differences as follows
$$
 \widetilde{\Delta}_h^1 (f):= \widetilde{\Delta}_h^1 (f,t)=  F_h (f,t)-f(t)=(F_h-\mathbb{I}) (f,t),
$$
$$
  \widetilde{\Delta}_h^m (f):=\widetilde{\Delta}_h^m (f,t)=\widetilde{\Delta}_h^1(\Delta_h^{m-1}(f),t)=(F_h-\mathbb{I})^m (f,t)=
  \sum_{k=0}^m k^{m-k} {m \choose k} F_{h,k} (f,t),
$$
where $m=2,3,\ldots$, $F_{h,0}(f):=f,$ $F_{h,k}(f):=F_h(F_{h,k}(f))$ and $\mathbb{I}$ is the identity operator in $B{\mathcal S}_{\bf M}.$ Consider the following smoothness characteristics
\begin{equation}\label{tilde-omega-def}
    \widetilde{\omega}_m (f, \delta):= \sup_{0\le h \le \delta}
    \|\widetilde{\Delta}_h^m (f) \|_{_{\scriptstyle {\bf M}}}, \quad \delta>0.
\end{equation}
It can be shown \cite{Abilov+Abilova-Math-note-2004} that $\omega_{\tilde{\varphi}_m}(f,\delta)_{_{\scriptstyle {\bf M}}} =
 \widetilde{\omega}_m(f,\delta)_{_{\scriptstyle {\bf M}}}$ for  for $\tilde{\varphi}_m(t)=(1-\mathrm{sinc}~ t)^m,$ $m \in \mathbb{N},$ where  $\mathrm{sinc}~ t=\{\sin t/t, ~ \mbox{when} ~t\not=0, \quad 1,  ~\mbox{when} ~t=0\}.$

In the general case,  moduli similar to (\ref{general_modulus})  were studied  in
\cite{Boman_1980, Vasil'ev_2001, Kozko_Rozhdestvenskii_2004, Vakarchuk_2016, Babenko_Savela_2012, Abdullayev_Chaichenko_Shidlich_2021, Abdullayev_Serdyuk_Shidlich_2021, Abdullayev_Chaichenko_Shidlich_2021_RMJ, Serdyuk-Shidlich_2022}
etc.



\section{Main results}

\subsection{Jackson type inequalities}
In this subsection, direct theorems are established for functions $f \in B{\mathcal S}_{\bf M}$  in terms of the best
approximations and  generalized moduli of smoothness. In particular, for functions $  f \in B{\mathcal S}_{\bf M}$ with the Fourier series of the form (\ref{Fourier_Series}), we prove Jackson type inequalities of the kind as
$$
    E_{\lambda_n}(f)_{_{\scriptstyle {\bf M}}}\le
    K(\tau )\omega_\varphi
    \Big(f, \frac {\tau }{\lambda_n}\Big)_{_{\scriptstyle {\bf M}}},
    \quad \tau >0, \quad n\in {\mathbb N}.
$$

Let $V(\tau)$, $\tau>0$, be a set of bounded nondecreasing functions $v $ that differ from a constant on $[0, \tau].$

 \begin{theorem}    \label{Th.1}
   Assume that the function $f\in B{\mathcal S}_{\bf M}$  has the Fourier series of the form (\ref{Fourier_Series}). Then for any  $\tau >0$, $n\in {\mathbb N}$ and $\varphi\in \Phi$ the following inequality holds:
\begin{equation}\label{En<omega}
    E_{\lambda_n}(f)_{_{\scriptstyle {\bf M}}}\le
    K_{n,\varphi}(\tau )
    \omega_\varphi\Big(f, \frac {\tau }{\lambda_n}\Big)_{_{\scriptstyle {\bf M}}},
  \end{equation}
where
\begin{equation}\label{(6.8)}
    K_{n,\varphi}(\tau ):=\inf\limits _{v  \in  V(\tau )}
    \frac {v  (\tau )-v  (0)}  {I_{n,\varphi}(\tau ,v  )},
\end{equation}
 and
\begin{equation}\label{(6.9)}
    I_{n,\varphi}(\tau ,v  ):= \inf\limits _{{k \in {\mathbb N}},\,k \ge n}
    \int\limits _0^{\tau }\varphi
    \Big(\frac{\lambda_k t}{\lambda_n} \Big){\mathrm d}v  (t).
\end{equation}
 Furthermore, there exists a function $v _*\in V(\tau )$ that realizes the greatest lower bound in (\ref{(6.8)}).
\end{theorem}

In the spaces  $L_2$ of $2\pi$-periodic square-summable
functions, for moduli of continuity $\omega_m(f;\delta)$ and $\tilde{\omega}_m (f;\delta)$, such result was obtained by  Babenko \cite{Babenko_1986}, and  Abilov and Abilova \cite{Abilov+Abilova-Math-note-2004}, respectively.
In the spaces ${\mathcal S}^p$  of functions of one and several variables, this result for classical moduli of smoothness was obtained in \cite{Stepanets_Serdyuk_2002} and   \cite{Abdullayev_Ozkartepe_Savchuk_Shidlich_2019}, respectively. In the Musielak-Orlicz  spaces ${\mathcal S}_{\bf M}$, similar result was obtained for generalized moduli of smoothness  in \cite{Abdullayev_Chaichenko_Shidlich_2021}.

In the Besicovich-Stepanets spaces $B{\mathcal S}^p$, a similar theorem  was proved in \cite{Serdyuk-Shidlich_2022}. It was noted above that  in the case when all  functions $M_k(u)=u^p\Big(p^{-1/p}q^{-1/p'}\Big)^p$, $p>1$, $1/p+1/p'=1$, we have $B{\mathcal S}_{\bf M}=B{\mathcal S}^{p}$ and $ \|f\|_{_{\scriptstyle  {\bf M}}}=\| f\|_{_{\scriptstyle  B{\mathcal S}^p}}$. In the case $p=1$, the similar equalities $B{\mathcal S}_{\bf M}=B{\mathcal S}^{1}$ and  $ \|f\|_{_{\scriptstyle  {\bf M}}}=\| f\|_{_{\scriptstyle  B{\mathcal S}^1}}$  obviously can be obtained if all $M_k(u)=u$, $k\in {\mathbb Z}$, and the set
$\Gamma$ is a set of all sequences of  positive numbers $\gamma=\{\gamma_k\}_{k\in \mathbb{Z}}$ such that  $\|\gamma\|_{l_\infty({\mathbb Z})}=\sup_{k\in \mathbb{Z}}\gamma_k \le 1$. Comparing estimate (\ref{En<omega}) with the corresponding result of Theorem 1  from \cite{Serdyuk-Shidlich_2022}, we see that in the case when $B{\mathcal S}_{\bf M}=B{\mathcal S}^{1}$, the inequality  (\ref{En<omega}) is unimprovable on the set of all functions
      $f\in B{\mathcal S}^{1}$, $f\not \equiv {\rm const}$. Furthermore,   Theorem 1 \cite{Serdyuk-Shidlich_2022} implies the existence of the function $v _*\in V(\tau )$ that realizes the greatest lower bound in (\ref{(6.8)}).

\begin{proof}

In the proof of Theorem \ref{Th.1}, we mainly use the ideas outlined in
\cite{Chernykh_1967, Chernykh_1967_MZ, Babenko_1986, Stepanets_Serdyuk_2002, Serdyuk-Shidlich_2022},   taking into account
the peculiarities of the spaces $B{\mathcal S}_{\bf M}$.  From 
 (\ref{def-Orlicz-norm}) and (\ref{Best_Approximation_Deff}), it follows that for any   $f\in B{\mathcal S}_{\bf M}$
with  the Fourier series of the form (\ref{Fourier_Series}), we have
  \begin{equation} \label{Best_Approx}
     E_{\lambda_n}(f)_{_{\scriptstyle \mathbf{M}}}=\|f-S_n(f)\|_{\bf M} =
     \sup \Bigg\{ \sum\limits_{|k|\ge n} \gamma_k |A_k(f)|:~ \gamma \in \Gamma \Bigg\}
 \end{equation}
where $S_n(f):=\sum _{|k|<n} A_{k}(f) {\mathrm e}^{ {\mathrm i}\lambda_k x}$.

By the definition of supremum, for  arbitrary $\varepsilon>0$  there exists a sequence $\tilde{\gamma} \in \Gamma$,
$\tilde{\gamma}=\tilde{\gamma}(\varepsilon)$, such that the following relations holds:
\[
    \sum\limits_{|k| \ge n} \tilde{\gamma}_k |A_k(f)|+\varepsilon  \ge \sup \Big\{ \sum\limits_{|k| \ge n}
    \gamma_k |A_k(f) |: \   \gamma\in \Gamma\Big\}.
\]

 For arbitrary $\varphi\in \Phi$ and $h\in {\mathbb R}$, 
 consider the sequence  of numbers $ \{\varphi(\lambda_kh) A_k(f) \}_{k\in {\mathbb Z}}$. If there exists a function
$\Delta_h^\varphi (f)\in B$-a.p. such that for all $k\in {\mathbb Z}$
  \begin{equation}\label{modulus_generalize difference_Fourier_Coeff}
     A_k(\Delta_h^\varphi (f))=\varphi(\lambda_kh) A_k(f),
 \end{equation}
then here and below we denote by  $\|\Delta_h^\varphi (f)\|_{\bf M}$
the Orlicz norm  (\ref{def-Orlicz-norm}) of the function  $\Delta_h^\varphi (f)$. If such a $B$-a.p function $\Delta_h^\varphi (f)$
does not exist, then to simplify notation we  also use the notation $\|\Delta_h^\varphi (f)\|_{\bf M}$,
meaning by it the $l_{\bf M}$-norm of the sequence
$ \{\varphi(\lambda_kh) A_k(f) \}_{k\in {\mathbb Z}}$. In view of (\ref{def-Orlicz-norm}) and (\ref{modulus_generalize difference_Fourier_Coeff}), we have
\[
    \|\Delta _h^\varphi f\|_{_{\scriptstyle  {\bf M}}} \ge
    \sup \Big\{ \sum\limits_{|k|\ge n}
    \gamma_k\varphi(\lambda_k h)| A_k(f) |: \   \gamma\in \Gamma \Big\}
    \ge \sum\limits_{|k|\ge n}
    \tilde{\gamma}_k \varphi(\lambda_k h)|A_k(f)|
\]
\[
   = \frac{I_{n,\varphi}(\tau ,v   )}{v (\tau)-v (0)}
   \sum\limits_{|k|\ge n} \tilde{\gamma}_k |A_k (f) |+
   \sum\limits_{|k|\ge n} \tilde{\gamma}_k |A_k (f)|
   \Big(\varphi(\lambda_k h)-\frac{I_{n,\varphi}(\tau ,v)}{v (\tau)-v (0)}\Big).
\]
For  any $ u\in [0,\tau]$, we get
 \begin{equation}\label{diff_est}
    \|\Delta _{\frac{u}{\lambda_n}}^\varphi f \|_{_{\scriptstyle  {\bf M}}} \ge
     \frac{I_{n,\varphi}(\tau ,v)}{v (\tau)-v (0)}
     \sum\limits_{|k|\ge n} \tilde{\gamma}_k |A_k(f) |
$$
$$
     +\sum\limits_{|k|\ge n} \tilde{\gamma}_k|A_k(f) |
     \bigg(\varphi\Big(\frac {\lambda_k u}{\lambda_n}\Big)-
     \frac{I_{n,\varphi}(\tau ,v   )}{v (\tau)-v (0)}\bigg).
 \end{equation}
The both sides of inequality (\ref{diff_est}) are nonnegative  and, in view of the
boundedness of the function $\varphi$, the series on its right-hand side is majorized on the entire
real axis by the absolutely convergent series
  $K(\varphi)\sum_{|k|\ge n} \tilde{\gamma}_k|A_k(f) |$, where
  $K(\varphi):=\max_{u\in {\mathbb R}} \varphi(u)$. Then
 integrating this inequality with respect to $dv  (u)$  from  $0$  to $\tau,$ we get
\[
    \int\limits_0^{\tau } \|\Delta_{\frac {u}{\lambda_n}}^\varphi f
    \|_{_{\scriptstyle  {\bf M}}}{\mathrm d}v(u)
    \ge I_{n,\varphi}(\tau ,v) \sum\limits_{|k|\ge n} \tilde{\gamma}_k |A_k(f)|
\]
\[
    +\sum\limits_{|k|\ge n} \tilde{\gamma}_k |A_k(f)|
    \bigg(\int\limits_0^{\tau } \varphi\Big(\frac{\lambda_k u}{\lambda_n}\Big){\mathrm d}v(u) - I_{n,\varphi}(\tau ,v) \bigg).
 \]

By virtue of the definition of $I_{n,\varphi}(\tau ,v   )$, we see that the second term on the right-hand side of
the last relation is nonnegative. Therefore, for any function $v  \in  V(\tau )$,  we have
\[
    \int\limits_0^{\tau }\|\Delta _{\frac{u}{\lambda_n}}^\varphi f
    \|_{_{\scriptstyle  {\bf M}}}{\mathrm d}v(u)  \ge
    I_{n,\varphi}(\tau ,v)  \sum\limits_{|k|\ge n}
    \tilde{\gamma}_k |A_k(f) |
\]
\[
    \ge I_{n,\varphi}(\tau ,v) \bigg( \sup \Big\{ \sum\limits_{|k| \ge n}
    \gamma_k |A_k(f) |: \   \gamma\in \Gamma\Big\} -\varepsilon \bigg),
\]
wherefrom due to an arbitrariness of choice of the number  $\varepsilon$, we conclude that
 \[
    \int\limits_0^{\tau }\|\Delta _{\frac{u}{\lambda_n}}^\varphi f
    \|_{_{\scriptstyle  {\bf M}}}{\mathrm d}v(u)  \ge
     I_{n,\varphi}(\tau ,v) E_{\lambda_n} (f)_{_{\scriptstyle  {\bf M}}}.
\]
 Hence,
 \begin{equation}\label{inequal-En-int}
   E_{\lambda_n} (f)_{_{\scriptstyle  {\bf M}}} \le \frac{1}{I_{n,\varphi}(\tau ,v)}
    \int\limits_0^{\tau }\|\Delta _{\frac{u}{\lambda_n}}^\varphi f
    \|_{_{\scriptstyle  {\bf M}}}{\mathrm d}v(u) \le
    \frac {1}{I_{n,\varphi}(\tau ,v   )} \int\limits_0^{\tau }
    \omega_\varphi \Big(f,\frac{u}{\lambda_n}\Big)_{_{\scriptstyle  {\bf M}}} {\mathrm d}v(u),
 \end{equation}
whence  taking into account nondecreasing of the function $\omega_\varphi$,
we immediately obtain relation (\ref{En<omega}).

\end{proof}

Now we consider some realisations of Theorem  \ref{Th.1}. Setting   $\varphi_\alpha(t) = 2^{\frac {\alpha   }2} (1-\cos t)^{\frac \alpha 2}$, $\alpha>0,$  $ \omega_{\varphi_\alpha}(f,\delta)_{_{\scriptstyle  {\bf M}}}=: \omega_\alpha(f,\delta)_{_{\scriptstyle  {\bf M}}}$, $\tau=\pi,$  and $v(u)=1-\cos u,$ $u\in [0,\pi],$   we get the following assertion.

 \begin{corollary}\label{Cor.1}
    For arbitrary numbers  $n\in {\mathbb N}$ and $\alpha>0$, and for any function $f\in B{\mathcal S}_{\mathbf{M}}$
    with the Fourier series of the form  (\ref{Fourier_Series}),  the following inequalities holds:
\begin{equation} \label{cor-inequal-omega-m}
    E_{\lambda_n} (f)_{_{\scriptstyle  {\bf M}}} \le
    \frac {1}{2^{\alpha\over 2}I_{n}({\alpha\over 2})} \int\limits_0^{\pi }
    \omega_\alpha \Big(f,\frac{u}{\lambda_n}\Big)_{_{\scriptstyle  {\bf M}}}
    \sin u \, {\mathrm d} u,
\end{equation}
where
\begin{equation} \label{In_1}
 I_{n}\Big({\alpha\over 2}\Big)= \inf_{k \in \mathbb{N}, k\ge n} \int_0^\pi
    \Big(1-\cos \frac{\lambda_k u}{\lambda_n}\Big)^{\alpha\over 2} \sin u \, {\mathrm d} u.
\end{equation}
If, in addition $\frac {\alpha }2\in {\mathbb N}$, then
\begin{equation} \label{In_12}
    I_n\Big(\frac {\alpha }2\Big)=\frac {2^{\frac { \alpha}2+1}}{\frac {\alpha }2+1},
\end{equation}
and the inequality (\ref{cor-inequal-omega-m}) cannot be improved for any
$n\in {\mathbb N}.$
 \end{corollary}

 \begin{proof} Estimate (\ref{cor-inequal-omega-m}) follows from (\ref{inequal-En-int}).  In \cite[relation (52)]{Stepanets_Serdyuk_2002}, it was shown that for any  $\theta\ge 1$ and $s\in {\mathbb N}$
the following inequality holds:
 \[
 \int_0^\pi (1-\cos\theta t)^s \sin t {\mathrm d}t\ge  \frac {2^{s+1}}{s+1},
 \]
which turns into equality for $\theta=1$. Therefore,  setting  $s={\alpha\over 2}$ and $\theta=\frac {\lambda_\nu}{\lambda_n}$, $\nu=n,n+1,\ldots$, and  the monotonicity  of the sequence of Fourier exponents $\{\lambda_k\}_{k\in {\mathbb Z}}$,
we see that for $\frac {\alpha}2\in {\mathbb N}$, indeed, the equality (\ref{In_12}) holds.

Let us prove that in this case, the constant $\frac {{\alpha\over 2}+1}{2^{\alpha+1}}$ in  inequality (\ref{cor-inequal-omega-m}) is unimprovable for
$\frac {\alpha }2\in {\mathbb N}$. It suffices to verify that  the function
\begin{equation}\label{A6.53}
    f^*(x)=\gamma +\beta e^{-\lambda_n x} + \delta e^{\lambda_n x},
\end{equation}
where $\gamma$, $\beta$ and $\delta $ are arbitrary complex numbers, satisfies the equality
\begin{equation}\label{(6.53)}
    E_{\lambda_n} (f^*)_{_{\scriptstyle  {\mathbf{M}}}} =
    \frac {\frac { \alpha}{2}+1}{2^{ \alpha+1}}\int\limits _0^{\pi }
    \omega _\alpha\Big(f^*, \frac t{\lambda_n}\Big)_{_{\scriptstyle  {\mathbf{M}}}}
    \sin t\, {\mathrm d}t, \ \  \   \alpha >0.
\end{equation}
We have
$
    E_{\lambda_n} (f^*)_{_{\scriptstyle  {\mathbf{M}}}} =|\beta |+|\delta |,
$
the function
$
    \|\Delta_{{u}/{\lambda_n}}^ {\varphi_\alpha}  f^*\|_{_{\scriptstyle  {\mathbf{M}}}}
    =2^{\frac {\alpha}2}(|\beta |+|\delta |) (1-\cos u)^{\frac { \alpha}2}
 $
does not decrease with respect to $u$ on $[0,\pi ].$ Therefore,
$
    \omega _\alpha  (f^*, \frac {u}{\lambda_n} )_{_{\scriptstyle  {\mathbf{M}}}} =
    \|\Delta_{u/\lambda_n}^{\varphi_\alpha}  f^*\|_{_{\scriptstyle  {\mathbf{M}}}},
$
and
\[
    \frac {2^{ \alpha+1}} {\frac { \alpha }2+1}E_{\lambda_n} (f^*)_{_{\scriptstyle{\mathbf{M}}}} -\int\limits _0^{\pi } \omega _\alpha
    \Big(f^*, \frac t{\lambda_n}\Big)_{_{\scriptstyle  {\mathbf{M}}}}
    \sin t\, {\mathrm d}t
\]
\[
    =(|\beta |+|\delta |)\Big(\frac {2^{\alpha+1}}{\frac {\alpha }2+1}-
    2^{\frac { \alpha}2}\int_0^\pi (1-\cos  t)^{\frac {\alpha}2} \sin t \,{\mathrm d}t\Big)=0
\]

\end{proof}

It was shown in   \cite{Stepanets_Serdyuk_2002} that  $I_n(s)\ge 2$ when $s\ge 1$
and  $I_n(s)\ge 1+2^{s-1}$ when $s\in (0,1)$.
Combining these two estimates and (\ref{cor-inequal-omega-m}), we obtain the following statement, which establishes a Jackson-type inequality with a constant uniformly bounded in the parameter   $n\in {\mathbb N}$.

\begin{corollary} \label{Theorem 2.4.}
   Assume that the function   $f\in B{\mathcal S}_{\mathbf{M}}$ has the Fourier series of the form (\ref{Fourier_Series})  and  $\|f-A_0(f)\|_{_{\scriptstyle  {\bf M}}} \not= 0$. Then for any   $n\in {\mathbb N}$ and $\alpha>0$,
  \begin{equation}\label{(6.16)}
       E_{\lambda_n} (f)_{_{\scriptstyle  {\bf M}}} <  \frac {4}{3\cdot 2^{\alpha/2}}\omega_\alpha
  \Big(f, \frac {\pi }{\lambda_n}\Big)_{_{\scriptstyle  {\bf M}}}.
  \end{equation}
  Furthermore, in the case where $\alpha=m\in {\mathbb N}$, the following more accurate estimate holds:
  \begin{equation}\label{(6.161)}
       E_{\lambda_n} (f)_{_{\scriptstyle  {\bf M}}} < \frac {4-2\sqrt{2}}{2^{m/2}}\omega_m
  \Big(f, \frac {\pi }{\lambda_n}\Big)_{_{\scriptstyle  {\bf M}}}.
   \end{equation}
\end{corollary}

\begin{proof}    Relation (\ref{(6.161)}) follows from the estimate $I_n(\frac{\alpha}2)\ge 1+\frac 1{\sqrt{2}}$, which is a consequence of the above estimates for the value of $I_n(s)$
in the case   $\alpha=m\in {\mathbb N}$  \cite{Stepanets_Serdyuk_2002}.
 \end{proof}


If the weight function $v  _2(t)=t$, then we obtain the following assertion:
 \begin{corollary} \label{Cor.2}
Assume that the function $f\in B{\mathcal S}_{\scriptstyle{\mathbf{M}}}$
has the Fourier series of the form (\ref{Fourier_Series}) and $\alpha\ge 1$.
Then for any  $0<\tau \le \frac {3\pi }4$ and $n \in {\mathbb N}$,
\begin{equation}\label{cor-2-estim-omega}
    E_{\lambda_n} (f)_{_{\scriptstyle  {\mathbf{M}}}} \le
    \frac {1}{2^{\alpha}\int _0^{\tau }\sin ^{\alpha}\frac t{2}\, {\mathrm d}t}
    \int\limits_0^{\tau} \omega_\alpha
    \Big(f, \frac t{\lambda_n}\Big)_{_{\scriptstyle{\mathbf{M}}}}\, {\mathrm d}t.
      \end{equation}
Relation  (\ref{cor-2-estim-omega}) becomes equality for the function $f^*$  of the form (\ref{A6.53}).
\end{corollary}

Inequalities (\ref{cor-inequal-omega-m}) and (\ref{cor-2-estim-omega}) can be considered as an extension of the corresponding results of  Serdyuk and Shidlich \cite{Serdyuk-Shidlich_2022} to the Besicovitch-Musielak spaces $B{\mathcal S}_{\scriptstyle{\mathbf{M}}}$, and  they coincide with them in the case $B{\mathcal S}_{\scriptstyle{\mathbf{M}}}=B{\mathcal S}^1.$  In the spaces  ${\mathcal S}^p$  of functions of one and several variables, analogues of Theorem  \ref{Th.1} and Corollaries  \ref{Cor.1} and \ref{Cor.2}  were proved in \cite{Stepanets_Serdyuk_2002} and  \cite{Abdullayev_Ozkartepe_Savchuk_Shidlich_2019}, respectively.  The inequalities of this type were also investigated in   \cite{Chernykh_1967_MZ, Vasil'ev_2001, Stepanets_Serdyuk_2002, Vakarchuk_2016, Babenko_Savela_2012}, etc.

\begin{proof} From inequality (\ref{inequal-En-int}), it follows that
\[
     E_{\lambda_n}(f)_{_{\scriptstyle \mathbf{M}}}
     \le \frac {1}{2^{ \alpha\over 2} I_n^*(\frac { \alpha}2)}\int\limits _0^{\tau }
     \omega _\alpha \Big(f, \frac t{\lambda_n}\Big)\,{\mathrm d}t,
\]
where
\[
    {I}_n^*\Big(\frac {\alpha}2\Big):=\inf\limits _{{k \in {\mathbb N}},\,k \ge n}
    \int\limits _0^{\tau } \Big(1-\cos \frac {\lambda_k t}{\lambda_n}\Big)^{\frac {\alpha}2}
    {\mathrm d}t, \ \  \alpha> 0, \ \ n\in {\mathbb N}.
\]

In \cite{Voicexivskij_2002}, it is shown that for the function
$F_\alpha(x):=\frac {1}{x} \int_0^x |\sin t|^\alpha \, {\mathrm d}t,$ any
$h\in (0,\frac{3\pi}4)$ and $\alpha \ge 1$, the following relation is true:
\begin{equation}\label{W1}
    \inf\limits_{x\ge h/2} F_\alpha (x)=F_\alpha (h/2).
\end{equation}
Since for   $h =\frac{\lambda_k}{\lambda_n}\ge 1$ ($k\ge n$)
\[
    \int\limits _0^{\tau } \Big(1-\cos \frac{\lambda_k t}{\lambda_n} \Big)^{\alpha\over 2}
    \,{\mathrm d}t= 2^{\alpha\over 2}\int\limits _0^{\tau }\Big|
    \sin \frac{\lambda_k t}{2 \lambda_n}\Big|^{\alpha}\,{\mathrm d}t=
    2^{\alpha\over 2}\tau F_{\alpha}\Big( {\frac{\lambda_k \tau}{2\lambda_n}}\Big),
\]
from (\ref{W1}) (with $\tau \in (0, \frac {3\pi }4]$ and $\alpha \ge 1$) we obtain
\[
    {I}_n^* \Big(\frac {\alpha}{2} \Big)= \inf\limits _{{k \in {\mathbb N}}:k  \ge n}
    \int\limits _0^{\tau }\Big(1 - \cos\frac {\lambda_k t}{\lambda_n}\Big)^{\frac {\alpha}2} {\mathrm d}t= \inf\limits _{{k \in {\mathbb N}}:k  \ge n}\!\!2^{\alpha\over 2}\!\!
    \int\limits _0^{\tau }\Big|\sin \frac {\lambda_k t}{2\lambda_n}\Big|^{\alpha}{\mathrm d}t =\! 2^{\alpha\over 2}\!\!\int\limits _0^{\tau } \sin^{\alpha} \frac {t}{2} \,{\mathrm d}t.
\]

For the functions  $f^*$  of the form (\ref{A6.53}), the equality
\[
    E_{\lambda_n} (f^*)_{_{\scriptstyle  {\mathbf{M}}}} =
    \frac {1}{2^{\alpha}\int _0^{\tau }\sin ^{\alpha}\frac t{2} \,{\mathrm d}t} \int\limits_0^{\tau} \omega_\alpha
    \Big(f^*, \frac t{\lambda_n}\Big)_{_{\scriptstyle{\mathbf{M}}}}\, {\mathrm d}t.
\]
is verified similarly to the proof of equality (\ref{(6.53)}).
\end{proof}


In the case $\varphi(t)=\tilde{\varphi}_m(t)=(1-\mathrm{sinc}~ t)^m,$ $m \in \mathbb{N},$ where, by definition, $\mathrm{sinc}~ t=\{\sin t/t,~ if ~t\not=0; \quad 1, ~if ~t=0\},$ for $\tau=\pi$ and $v(u)=1-\cos u,$ $u\in [0;\pi],$ from relation (\ref{inequal-En-int}) we get
$$
    E_{\lambda_n} (f)_{_{\scriptstyle  {\bf M}}} \le
    \frac {1}{\tilde {I}_{n}(m)} \int\limits_0^{\pi }
    \tilde{\omega}_m \Big(f,\frac{u}{\lambda_n}\Big)_{_{\scriptstyle  {\bf M}}}
    \sin u \,{\mathrm d} u,
$$
where
$$
    \tilde{I}_{n}(m)= \inf_{k \in \mathbb{N}, k\ge n} \int_0^\pi
    \Big(1-\mathrm{sinc} ~\frac{\lambda_k u}{\lambda_n}\Big)^{m}
    \sin u \,{\mathrm d} u.
$$

Taking into account the estimation \cite{Vakarchuk+Zabutnaya-Math-notes-2016}
$$
    1- \mathrm{sinc}~\Big(\frac{\lambda_k u}{\lambda_n}\Big) \ge 1-\frac{\sin u}{u}\ge \Big(\frac{u}{\pi}\big)^2, \quad k\ge n, \quad u \in [0;\pi],
$$
we have
$$
    \tilde{I}_{n}(m)\ge \int_0^\pi
    \Big(1-\mathrm{sinc} ~ u \Big)^{m}   \sin u \,{\mathrm d} u \ge
    \frac{1}{\pi^{2m}}\int_0^\pi u^{2m} \sin u \,{\mathrm d} u
$$
$$
    =\frac{2m!}{\pi^{2m}} \Big(\sum_{j=0}^m (-1)^j \frac{\pi^{2m-2j}}{(2m-2j)!}+
    \frac{\pi^{2m}}{2m!}(-1)^m \Big):=\frac{2m!}{\pi^{2m}} K(m).
$$

Thereby, the following corollary follows from Theorem \ref{Th.1}.

 \begin{corollary}\label{Cor.3}
    For arbitrary numbers  $n\in {\mathbb N}$ and $m>0$, and for any function $f\in B{\mathcal S}_{\mathbf{M}},$
    with the Fourier series of the form  (\ref{Fourier_Series}) the following inequalities holds:
\begin{equation} \label{cor-inequal-Omega-m}
    E_{\lambda_n} (f)_{_{\scriptstyle  {\bf M}}} \le
    \frac {\pi^{2m}}{2m! K(m)} \int\limits_0^{\pi }
    \tilde{\omega}_m \Big(f,\frac{u}{\lambda_n}\Big)_{_{\scriptstyle  {\bf M}}}
    \sin u \, {\mathrm d} u,
\end{equation}
where
$$
    K(m)= \sum_{j=0}^m (-1)^j \frac{\pi^{2m-2j}}{(2m-2j)!}+
    \frac{\pi^{2m}}{2m!}(-1)^m.
$$
\end{corollary}

In the case $m=1,$  we have $2K(1)=\pi^2-4$ and
$$
    E_{\lambda_n} (f)_{_{\scriptstyle  {\bf M}}} \le
    \frac {\pi^{2}}{\pi^2-4} \int\limits_0^{\pi }
    \tilde{\omega}_1 \Big(f,\frac{u}{\lambda_n}\Big)_{_{\scriptstyle  {\bf M}}}
    \sin u \, {\mathrm d} u \le
    \frac {\pi^2 \lambda_n }{\pi^2-4} \int\limits_0^{\frac{\pi}{\lambda_n} }
    \tilde{\omega}_1 \Big(f,u\Big)_{_{\scriptstyle  {\bf M}}}
    \sin \lambda_n u \, {\mathrm d} u.
$$

If the weight function $v  _2(t)=u^{m+1}$, then we obtain the following assertion:
 \begin{corollary} \label{Cor-tilde-omega-tm}
Assume that the function $f\in B{\mathcal S}_{\scriptstyle{\mathbf{M}}}$
has the Fourier series of the form (\ref{Fourier_Series}) and $m \ge 1$.
Then for any  $0<\tau \le \pi$ and $n \in {\mathbb N}$,
\begin{equation}\label{inequal-tilde-omega}
    E_{\lambda_n} (f)_{_{\scriptstyle  {\mathbf{M}}}} \le
    \pi^{m-1} \Big(\frac {2 \lambda_n }{\pi^2-4}\Big)^m \lambda_n
    \int\limits_0^{\tau/\lambda_n} \tilde{\omega}_m
    (f, t)_{_{\scriptstyle{\mathbf{M}}}} t^{m} \, {\mathrm d}t.
      \end{equation}
\end{corollary}

Ideed, applying Holder's inequality, we find
$$
    \int_0^\pi \Big(1-\mathrm{sinc} ~ \frac{\lambda _k u}{\lambda_n} \Big)^{m}  \,{\mathrm d} u^{m+1} \ge
    (m+1) \int_0^\pi \Big(1-\frac{\sin  u}{u} \Big)^m u^{m} \,{\mathrm d} u
$$
$$
    =(m+1) \int_0^\pi (u-\sin  u )^m  \,{\mathrm d} u \ge
    \frac{m+1}{\pi^{m-1}} \Big( \int_0^\pi (u-\sin  u )  \,{\mathrm d} u\Big)^m=
    \frac{m+1}{\pi^{m-1}} \Big( \frac{\pi^2-4}{2}\Big)^m.
$$

In the spaces  $L_2$ of $2\pi$-periodic square-summable
functions, for moduli of smoothness $\tilde{\omega}_m (f;\delta)$,  the results of this kind were obtained by Abilov and Abilova \cite{Abilov+Abilova-Math-note-2004}, and  Vakarchuk \cite{Vakarchuk_2016}. Note that   in the case $f\in B{\mathcal S}_{\scriptstyle{\mathbf{M}}}=L_2$ the inequality (\ref{inequal-tilde-omega}) follows from the result of \cite{Abilov+Abilova-Math-note-2004} (see Theorem 1). For $m=1$ and $f\in L_2$, the statements of Corollary \ref{Cor-tilde-omega-tm} and Theorem 1 from \cite{Abilov+Abilova-Math-note-2004} are identical, and the constant in the  right  side of (\ref{inequal-tilde-omega}) cannot be reduced for every fixed $n$.


\section{ Inverse approximation theorem.}

 \begin{theorem}\label{Inverse_Theorem}
Assume that   $ f\in B{\mathcal S}_{\bf M}$  has the Fourier series of the form (\ref{Fourier_Series}), the function $\varphi\in \Phi$ is  nondecreasing on an interval $[0,\tau]$ and $\varphi(\tau)=\max\{\varphi(t):t\in {\mathbb R}\}$. Then for any $n\in {\mathbb N}$, the following inequality holds:
\begin{equation}\label{S_M.12}
    \omega_\varphi \Big(f, \frac{\tau}{\lambda_n} \Big)_{_{\scriptstyle {\bf M}}}
    \le  \sum _{\nu =1}^{n} \Big(\varphi\Big(\frac {\tau \lambda_\nu}{\lambda_n}\Big)-
    \varphi\Big(\frac {\tau \lambda_{\nu-1}}{\lambda_n}\Big)\Big)
    E_{\lambda_\nu} (f)_{_{\scriptstyle {\bf M}}}.
\end{equation}
 \end{theorem}

\begin{proof}
Let us use the proof scheme from \cite{Stepanets_Serdyuk_2002} and \cite{Abdullayev_Chaichenko_Shidlich_2021}, modifying it taking into account the peculiarities of the spaces $B{\mathcal S}_{\bf M}$ and the definition of the modulus $\omega_\varphi$.

Let $ f\in B{\mathcal S}_{\bf M}$. For any $\varepsilon>0$ there exist a number $N_0=N_0(\varepsilon)\in {\mathbb N}$, $N_0>n$,
such that for any $N>  N_0$, we have
\[
    E_{\lambda_N}(f)_{_{\scriptstyle {\bf M}}}=
    \|f-{S}_{N-1}({f})\|_{_{\scriptstyle {\bf M}}} <\varepsilon/\varphi(\tau).
\]

Let us set $f_{0}:=S_{N_0}(f)$. Then in view of
(\ref{modulus_generalize difference_Fourier_Coeff}), we see that
\begin{equation}\label{(6.3100)}
    \|\Delta_h^\varphi (f)\|_{_{\scriptstyle {\bf M}}}
    \le \|\Delta_h^\varphi (f_{0})\|_{_{\scriptstyle {\bf M}}}
    +\|\Delta_h^\varphi (f-f_{0})\|_{_{\scriptstyle {\bf M}}}
$$
$$
    \le  \|\Delta_h^\varphi (f_0)\|_{_{\scriptstyle {\bf M}}}
    + \varphi(\tau) E_{\lambda_{N_0+1}}(f)_{_{\scriptstyle {\bf M}}}
    <\|\Delta_h^\varphi (f_0)\|_{_{\scriptstyle {\bf M}}}+\varepsilon.
\end{equation}
Further, let ${S}_{n-1}:={S}_{n-1}(f_0)$  be the Fourier sum of $f_0$.  Then  by virtue of (\ref{modulus_generalize difference_Fourier_Coeff}), for $|h|\le \tau /\lambda_n$, we have
\[
    \|\Delta_h^\varphi (f_0) \|_{_{\scriptstyle {\bf M}}}=\|\Delta_h^\varphi (f_0-S_{n-1})+\Delta_h^\varphi S_{n-1}\|_{_{\scriptstyle {\bf M}}}
    \le\Big\| \varphi(\tau)(f_0-S_{n-1})
\]
\begin{equation}\label{(6.73)}
    +
    \sum _{|k|\le n-1}\varphi(\lambda_kh)|A_k(f)| \Big\|_{_{\scriptstyle {\bf M}}}
     \le  \Big\|\varphi(\tau) \sum _{\nu =n}^{N_0} H_{\nu} +
     \sum _{\nu =1}^{n-1}\varphi\Big(\frac {\tau \lambda_\nu}{\lambda_n}\Big)
     H_{\nu}  \Big\|_{_{\scriptstyle {\bf M}}},
\end{equation}
where $H_{\nu}(x) :=|A_\nu(f)|+ |A_{-\nu}(f)|$, $\nu=1,2,\ldots$.

Now we use the following assertion from \cite{Stepanets_Serdyuk_2002}.

 \begin{lemma}[\cite{Stepanets_Serdyuk_2002}]
       \label{Lemma_31}
       Let  $\{c_{\nu}\}_{\nu=1}^\infty$ and $\{a_{\nu}\}_{\nu=1}^\infty$ be arbitrary numerical sequences.
       Then the following equality holds for  all natural $N_1$, $N_2$ and $N$ ${N_1}\le {N_2}<N$:
       \begin{equation}\label{(6.74)}
       \sum _{\nu ={N_1}}^{N_2}a_{\nu }c_{\nu }=a_{N_1}\sum _{\nu={N_1}}^{N }c_{\nu }+
       \sum _{\nu ={N_1}+1}^{N_2}(a_{\nu } -a _{\nu-1})\sum _{i=\nu }^{N }c_i-a_{N_2}\sum _{\nu ={N_2}+1}^{N }c_{\nu}.
       \end{equation}
 \end{lemma}

 \noindent Setting  $a_{\nu }=\varphi\Big(\frac {\tau \lambda_\nu}{\lambda_n}\Big),$
$c_{\nu }=H_{\nu}(x), $ ${N_1}=1$,  ${N_2}=n-1$ and $N=N_0$ in (\ref{(6.74)}), we get
\[
    \sum _{\nu =1}^{n-1}\varphi\Big(\frac {\tau \lambda_\nu}{\lambda_n}\Big)
    H_{\nu}(x)= \varphi\Big(\frac {\tau \lambda_1}{\lambda_n}\Big) \sum _{\nu =1}^{N_0}H_{\nu}(x)
\]
\[
   +\sum _{\nu =2}^{n-1}\bigg(\varphi\Big(\frac {\tau \lambda_\nu}{\lambda_n}\Big)-
   \varphi\Big(\frac {\tau \lambda_{\nu-1}}{\lambda_n}\Big)\bigg)
    \sum_{i=\nu }^{N_0}H_{i}(x) -\varphi\Big(\frac {\tau \lambda_{\nu-1}}{\lambda_n}\Big)
    \sum _{\nu =n}^{N_0 }H_{\nu}(x).
\]
Therefore,
\[
    \bigg\|\varphi(\tau) \sum _{\nu =n}^{N_0} H_{\nu} +
     \sum _{\nu =1}^{n-1}\varphi\Big(\frac {\tau \lambda_\nu}{\lambda_n} \Big)
     H_{\nu} \bigg\|_{_{\scriptstyle {\bf M}}}
\]
\[
   \le \bigg\|\varphi(\tau) \sum _{\nu =n}^{N_0} H_{\nu} +
   \sum _{\nu =1}^{n-1}\bigg(\varphi\Big(\frac {\tau \lambda_\nu}{\lambda_n}\Big)-
   \varphi\Big(\frac {\tau \lambda_{\nu-1}}{\lambda_n}\Big)\bigg)
   \sum_{i=\nu }^{N_0}H_{i}  -\varphi\Big(\frac {\tau \lambda_{\nu-1}}{\lambda_n}\Big)
   \sum_{\nu =n}^{N_0}H_{\nu}  \bigg\|_{_{\scriptstyle {\bf M}}}
\]
\begin{equation}\label{(6.74aqq)}
    \le \bigg\|\sum _{\nu =1}^{n}
    \bigg(\varphi\Big(\frac {\tau \lambda_\nu}{\lambda_n}\Big)-
    \varphi\Big(\frac {\tau \lambda_{\nu-1}}{\lambda_n}\Big)\bigg)
    \sum_{i=\nu }^{N_0}H_{i} \bigg\|_{_{\scriptstyle {\bf M}}}
    \le    \sum _{\nu =1}^{n}
    \bigg(\varphi\Big(\frac {\tau \lambda_\nu}{\lambda_n}\Big)-
    \varphi\Big(\frac {\tau \lambda_{\nu-1}}{\lambda_n}\Big)\bigg)
    E_{\lambda_\nu} (f_0)_{_{\scriptstyle {\bf M}}}.
\end{equation}

Combining relations (\ref{(6.3100)}), (\ref{(6.73)})  and (\ref{(6.74aqq)}) and taking into account the definition of the function $f_0$, we see that for  
$|h|\le \tau /\lambda_n$, the following inequality holds:
\[
    \|\Delta_h^\varphi (f)\|_{_{\scriptstyle {\bf M}}} \le
    \sum _{\nu =1}^{n}\Big(\varphi\Big(\frac {\tau \lambda_\nu}{\lambda_n}\Big)-
    \varphi\Big(\frac {\tau \lambda_{\nu-1}}{\lambda_n}\Big)\Big)
    E_{\lambda_\nu} (f)_{_{\scriptstyle {\bf M}}} +\varepsilon
\]
which, in view of arbitrariness of $\varepsilon$,  gives us (\ref{S_M.12}).

\end{proof}

Consider an important special case when
$\varphi(t)=\varphi_\alpha(t)=  2^{\frac {\alpha}2} (1-\cos t)^{\frac \alpha2}=2^\alpha  |\sin (t/2) |^\alpha $, $\alpha>0$.
In this case, the function  $\varphi$ satisfies the conditions of Theorem
 \ref{Inverse_Theorem} with   $\tau=\pi$.  Then  for $\alpha\ge 1$ using the inequality $x^\alpha-y^\alpha \le \alpha
 x^{\alpha-1}(x-y),$ $x>0, y>0$
(see, for example, \cite[Ch.~1]{Hardy_Littlewood_Polya_1934}), 
and the usual trigonometric formulas, for $\nu=1,2,\ldots,n,$ we get
\[
  \varphi \Big(\frac {\tau \lambda_\nu}{\lambda_n}\Big)-
  \varphi \Big(\frac {\tau \lambda_{\nu-1}}{\lambda_n}\Big)=
  2^{\alpha  }   \Big(\Big|\sin  \frac {\pi \lambda_\nu}{\lambda_n}  \Big|^{\alpha }  -
  \Big|\sin  \frac {\pi \lambda_{\nu-1}}{\lambda_n}  \Big|^{ \alpha}  \Big)
\]
\[
    \le 2^{\alpha }   {\alpha}   |\sin  \frac {\pi \lambda_\nu}{\lambda_n}  \Big|^{{\alpha}  -1}
    \Big|\sin  \frac {\pi \lambda_\nu}{\lambda_n}  -
  \sin  \frac {\pi \lambda_{\nu-1}}{\lambda_n}  \Big|
  \le {\alpha}   \Big(\frac{2\pi }{\lambda_n}\Big)^{\alpha}   \lambda_\nu^{{\alpha }  -1}(\lambda_{\nu}-\lambda_{\nu-1}).
\]
 If  $0<\alpha<1$, then the similar estimate can be obtained  using the inequality $x^\alpha-y^\alpha \le \alpha
  y^{\alpha-1}(x-y)$,
which holds for any $x>0, y>0,$ \cite[Ch.~1]{Hardy_Littlewood_Polya_1934}. 
 Hence, for any $ f\in B{\mathcal S}_{\bf M}$,  we get the following estimate:
       \begin{equation}\label{Inverse_Inequality_old}
       \omega _\alpha \Big(f, \frac{\pi}{\lambda_n}\Big)_{_{\scriptstyle {\bf M}}}\le
       \alpha \Big(\frac{2\pi }{\lambda_n}\Big)^\alpha
       \sum _{\nu =1}^{n}    \lambda_\nu^{{\alpha }  -1}(\lambda_{\nu}-\lambda_{\nu-1}) E_{\lambda_\nu} (f)_{_{\scriptstyle {\bf M}}}  \quad \alpha>0.
       \end{equation}
It should be noted that the constant in this estimate can be improved as follows.

 \begin{theorem}\label{Theorem_21a}
Assume that   $ f\in B{\mathcal S}_{\bf M}$  has the Fourier series of the form (\ref{Fourier_Series}). Then for any $n\in {\mathbb N}$ and $\alpha>0$,
\begin{equation}\label{S_M.12a}
    \omega_\alpha \Big(f, \frac{\tau}{\lambda_n} \Big)_{_{\scriptstyle {\bf M}}}
    \le
        \Big(\frac {\pi  }{\lambda_n}\Big)^\alpha
      \sum _{\nu =1}^{n}
     (\lambda_\nu^\alpha-\lambda_{\nu-1}^\alpha)
    E_{\lambda_\nu} (f)_{_{\scriptstyle {\bf M}}}.
\end{equation}
 \end{theorem}

\begin{proof}

We prove this theorem similarly to the proof of  Theorem \ref{Inverse_Theorem}.  For any $\varepsilon>0$, denote by $N_0=N_0(\varepsilon)\in {\mathbb N}$, $N_0>n$, a number such that for any $N>  N_0$
\[
    E_{\lambda_N}(f)_{_{\scriptstyle {\bf M}}}=
    \|f-{S}_{N-1}({f})\|_{_{\scriptstyle {\bf M}}} <\varepsilon.
\]
Let us set $f_{0}:=S_{N_0}(f)$, ${S}_{n-1}:={S}_{n-1}(f_0)$ and $\|\Delta_h^\alpha (f)\|_{_{\scriptstyle {\bf M}}}:=\|\Delta_h^{\varphi_\alpha}
(f)\|_{_{\scriptstyle {\bf M}}}$, and use relations (\ref{(6.3100)}) and (\ref{(6.73)}). We obtain
\begin{equation}\label{(6.3100a)}
    \|\Delta_h^\alpha (f)\|_{_{\scriptstyle {\bf M}}}
    <\|\Delta_h^\alpha (f_0)\|_{_{\scriptstyle {\bf M}}}+\varepsilon.
\end{equation}
and
\begin{equation}\label{(6.73a)}
    \|\Delta_h^\alpha (f_0)\|_{_{\scriptstyle {\bf M}}}\le     \bigg\|2^\alpha\sum _{\nu =n}^{N_0} H_{\nu} +
    2^\alpha \sum _{\nu =1}^{n-1}\Big|\sin  \frac {\pi \lambda_\nu}{2\lambda_n}\Big|^\alpha
     H_{\nu}  \bigg\|_{_{\scriptstyle {\bf M}}}
$$
$$
     \le    \Big(\frac {\pi  }{\lambda_n}\Big)^\alpha \bigg\|\lambda_n^\alpha\sum _{\nu =n}^{N_0} H_{\nu} +
     \sum _{\nu =1}^{n-1} \lambda_\nu^\alpha H_{\nu}  \bigg\|_{_{\scriptstyle {\bf M}}},
\end{equation}
where $|h|\le \pi /\lambda_n$ and $H_{\nu}(x)=|A_\nu(f)|+ |A_{-\nu}(f)|$, $\nu=1,2,\ldots$

By virtue of  (\ref{(6.74)}), for  $a_{\nu }= \lambda_\nu^\alpha,$
$c_{\nu }=H_{\nu}(x), $ ${N_1}=1$,  ${N_2}=n-1$ and $N=N_0$,
\[
    \sum _{\nu =1}^{n-1} \lambda_\nu^\alpha   H_{\nu}(x)= \lambda_1^\alpha \sum _{\nu =1}^{N_0}H_{\nu}(x)
   +\sum _{\nu =2}^{n-1}\Big(\lambda_\nu^\alpha-\lambda_{\nu-1}^\alpha\Big)
    \sum_{i=\nu }^{N_0}H_{i}(x) -\lambda_{\nu-1}^\alpha
    \sum _{\nu =n}^{N_0 }H_{\nu}(x).
\]
Therefore,
\begin{equation}\label{(6.74aqqa)}
\bigg\| \lambda_n^\alpha\sum _{\nu =n}^{N_0} H_{\nu} +
     \sum _{\nu =1}^{n-1} \lambda_\nu^\alpha H_{\nu}  \bigg\|_{_{\scriptstyle {\bf M}}}
=
 \bigg\|\sum _{\nu =1}^{n}
   \Big(\lambda_\nu^\alpha-\lambda_{\nu-1}^\alpha\Big)
    \sum_{i=\nu }^{N_0}H_{i} \bigg\|_{_{\scriptstyle {\bf M}}}
 $$
 $$
    \le    \sum _{\nu =1}^{n}
    \Big(\lambda_\nu^\alpha-\lambda_{\nu-1}^\alpha\Big)
    E_{\lambda_\nu} (f_0)_{_{\scriptstyle {\bf M}}}.
\end{equation}

Combining relations (\ref{(6.3100a)}), (\ref{(6.73a)})  and (\ref{(6.74aqqa)}) and taking into account the definition of the function $f_0$, we see that for 
 $|h|\le \tau /\lambda_n$, the following inequality holds:
\[
    \|\Delta_h^\alpha (f)\|_{_{\scriptstyle {\bf M}}} \le
    \Big(\frac {\pi  }{\lambda_n}\Big)^\alpha
      \sum _{\nu =1}^{n}
    \Big(\lambda_\nu^\alpha-\lambda_{\nu-1}^\alpha\Big)
    E_{\lambda_\nu} (f)_{_{\scriptstyle {\bf M}}} +\varepsilon
\]
which, in view of arbitrariness of $\varepsilon$,  gives us (\ref{S_M.12a}).

\end{proof}

In  (\ref{S_M.12}), the constant
 $\pi^\alpha$ is exact   in the sense that for any   $\varepsilon>0$, there exists a function $f^*\in B{\mathcal S}_{\bf M}$ such that for all $n$  greater that a certain number $n_0$, we have
 \begin{equation}\label{(6.31abcd)}
  \omega_{\alpha}^{\ }\Big(f^*, \frac {\pi }{\lambda_n}\Big)_{_{\scriptstyle {\bf M}}}>
 \frac {\pi ^\alpha-\varepsilon }{\lambda_n^\alpha }\sum _{\nu =1}^n\Big(\lambda_\nu^\alpha-\lambda_{\nu-1}^\alpha\Big) E_{\lambda_\nu}(f^*)_{_{\scriptstyle {\bf M}}}.
 \end{equation}
 Consider the function  $f^*(x)={\rm e}^{\mathrm{i}\lambda_{k_0}x}$, where $k_0$ is an arbitrary positive integer. Then   $E_{\lambda_\nu}(f^*)_{_{\scriptstyle {\bf M}}}=1$ for $\nu=1,2,\ldots,k_0$,  $E_{\lambda_\nu}(f^*)_{_{\scriptstyle {\bf M}}}=0$  for $\nu>k_0$ and
 \[
  \omega_{\alpha}^{\ }\Big(f^*, \frac {\pi }{\lambda_n}\Big)_{_{\scriptstyle {\bf M}}}\ge
 \|\Delta _{\frac {\pi }{\lambda_n}}^\alpha f^*\|_{_{\scriptstyle {\bf M}}} \ge
  2^\alpha\Big|\sin \frac {\lambda_{k_0} \pi}{2{\lambda_n}}\Big|^{\alpha  }.
 \]
 Since ${\sin t}/t$\ \  tends\ \  to $1$ as $t\to 0$, then for all $n$   greater that a certain number $n_0$, the  inequality   $2^\alpha |\sin {\lambda_{k_0} \pi}/{(2{\lambda_n})}|^\alpha>
 {(\pi^\alpha-\varepsilon)} \lambda_{k_0}^\alpha/{{\lambda_n}^\alpha}$ holds, which  yields  (\ref{(6.31abcd)}).

 \begin{corollary}
       \label{Corollary 21}
       Suppose  that  $ f\in B{\mathcal S}_{\bf M}$  has the Fourier series of the form (\ref{Fourier_Series}). Then for any  
       $n\in {\mathbb N}$  and  $\alpha>0$,
       \begin{equation}\label{Inverse_Inequality}
       \omega_{\alpha} \Big(f, \frac{\pi}{\lambda_n}\Big)_{_{\scriptstyle {\bf M}}}\le      \alpha  \Big(\frac{ \pi }{\lambda_n}\Big)^
       {\alpha}
       \sum _{\nu =1}^{n}  \lambda_\nu^{\alpha-1}(\lambda_{\nu}-\lambda_{\nu-1}) E_{\lambda_\nu}(f)_{_{\scriptstyle {\bf M}}}.
       \end{equation}
       If, in addition, the Fourier exponents  $\lambda_\nu$, $\nu\in {\mathbb N}$, satisfy the condition
       \begin{equation}\label{Lambda_Cond}
        \lambda_{\nu+1}-\lambda_{\nu} \le C ,\quad \nu=1,2,\ldots,
       \end{equation}
      with an absolute  constant $C>0$, then
       \begin{equation}\label{Inverse_Inequality_for_using}
       \omega_{\alpha}  \Big(f, \frac{\pi}{\lambda_n}\Big)_{_{\scriptstyle {\bf M}}}\le       C \alpha
        \Big(\frac {\pi  }{\lambda_n}\Big)^\alpha
       \sum _{\nu =1}^{n}  \lambda_\nu^{{\alpha  }-1}  E_{\lambda_\nu}  (f)_{_{\scriptstyle {\bf M}}}.
       \end{equation}
\end{corollary}


\section{Constructive characteristics of the classes of functions defined by the generalized moduli of smoothness}

Let  $\omega$ be the function (majorant) given on  $[0,1]$. For a fixed $\alpha>0$, we set
\begin{equation} \label{omega-class}
    B{\mathcal S}_{\bf M} H^{\omega}_{\alpha} =
    \Big\{f\in B{\mathcal S}_{\bf M} :  \quad \omega_\alpha(f, \delta)_{_{\scriptstyle {\bf M}}}=
    {\mathcal O}  (\omega(\delta)),\quad  \delta\to 0+\Big\}.
\end{equation}
Further, we consider the majorants   $\omega(\delta)$, $\delta\in [0,1]$, which satisfy the following conditions 1)--4): \noindent  {1)} $\omega(\delta)$ is continuous on $[0,1]$;\
 {  2)} $\omega(\delta)\uparrow$;\   {  3)}
$\omega(\delta)\not=0$ for $\delta\in (0,1]$;\   {  4)}~$\omega(\delta)\to 0$  for $\delta\to 0$; as well as the condition
\begin{equation} \label{B_alpha}
\quad \sum_{v=1}^n \lambda_v^{s-1}\omega\Big({1\over \lambda_v}\Big) =
{\mathcal O}  \Big[\lambda_n^s \omega \Big( {1\over \lambda_n}\Big)\Big].
\end{equation}
where $s>0$, and $\lambda_\nu$, $\nu\in {\mathbb N}$, is a increasing sequence of positive numbers.
In the case where $\lambda_\nu= \nu$, the condition $(\ref{B_alpha})$ is the known Bari condition
$({\mathscr B}_s)$ (see, e.g.
\cite{Bari_Stechkin_1956}).

 \begin{theorem}\label{Theorem 6.1}  Assume that the function $f\in B{\mathcal S}_{\bf M}$
  has the Fourier series of the form  (\ref{Fourier_Series}),
   $\alpha>0$ and the majorant  $\omega $ satisfies the conditions
       $1)$--\,$4)$.

       i) If  $f\in B{\mathcal S}_{\bf M}H^{\omega}_{\alpha}$, then the following relation is true:
     \begin{equation} \label{iff-theorem}
         E_{\lambda_n}(f)_{_{\scriptstyle {\bf M}}}={\mathcal O} \Big[ \omega \Big({1 \over {\lambda_n}} \Big) \Big].
      \end{equation}

      ii) If the numbers $\lambda_\nu$, $\nu\in {\mathbb N}$ satisfy condition $(\ref{Lambda_Cond})$ and
      the function $\omega  $  satisfies condition $(\ref{B_alpha})$ with $s=\alpha  $, then relation (\ref{iff-theorem})
      yields the inclusion    $f\in B{\mathcal S}_{\bf M}H^{\omega}_{\alpha}$.
\end{theorem}

\begin{proof} Let $f \in B{\mathcal S}_{\bf M}H^{\omega}_{\alpha}$. Then  relation   (\ref{iff-theorem})  follows from
  (\ref{omega-class}) and (\ref{(6.16)}).

 On the other hand, if
$f\in B{\mathcal S}_{\bf M}$, the numbers $\lambda_\nu$, $\nu\in {\mathbb N}$ satisfy condition $(\ref{Lambda_Cond})$ and
      the function $\omega  $  satisfies condition $(\ref{B_alpha})$ with $s=m  $, and relation
      (\ref{iff-theorem}) holds, then by (\ref{Inverse_Inequality_for_using}), we get
\[
    \omega_{\alpha} \Big(f, \frac {1 }{\lambda_n}\Big)_{_{\scriptstyle {\bf M}}}\le
     \frac{C_1 }{\lambda_n^\alpha}
       \sum _{\nu =1}^{n}  \lambda_\nu^{{m  }-1}  E_{\lambda_\nu}  (f)\le
\frac{C_1 }{\lambda_n^\alpha}
       \sum _{\nu =1}^{n}  \lambda_\nu^{{m  }-1}   \omega \Big({1\over {\lambda_\nu} }\Big)=
    {\mathcal O}  \Big[\omega  \Big( {1\over {\lambda_n}}\Big)\Big],
\]
where $C_1=m   (2\pi)^{m p}\cdot C$. Hence, the function    $f$ belongs to the set   $B{\mathcal S}_{\bf M}H^{\omega}_{\alpha}$.
\end{proof}

The function $t^r$, $0<r\le \alpha,$ satisfies condition   (\ref{B_alpha}) with $s=\alpha  $.
 Hence, denoting by $B{\mathcal S}_{\bf M}H_{\alpha}^r$ the class $B{\mathcal S}_{\bf M}H^{\omega}_{\alpha}$ for
  $\omega(t)=t^r$ we establish the following statement:

\begin{corollary}\label{corollary 6.1.} Let $f \in B{\mathcal S}_{\bf M}$  has the Fourier series of the form  (\ref{Fourier_Series}), $\alpha >0$,
$0<r\le \alpha $ and condition  (\ref{Lambda_Cond}) holds. The function  $f$ belongs to the set   $B{\mathcal S}_{\bf M}H_{\alpha}^r$, iff the  following relation is true:
$$
    E_{\lambda_n}(f)_{_{\scriptstyle {\bf M}}}={\mathcal O}   ({\lambda_n^{-r}} ).
$$
\end{corollary}

In the spaces  ${\mathcal S}^p$, for  classical moduli of smoothness $\omega_m$, Theorems \ref{Inverse_Theorem}
and \ref{Theorem 6.1} were proved in \cite{Stepanets_Serdyuk_2002} and
\cite{Abdullayev_Ozkartepe_Savchuk_Shidlich_2019}. In the spaces
${\mathcal S}^p$, inequalities of the form (\ref{Inverse_Inequality_for_using})  were also obtained in \cite{Sterlin_1972}.
In spaces $L_p$ of  $2\pi$-periodic  Lebesgue summable with the $p$th degree functions,   inequalities of the kind
as (\ref{Inverse_Inequality_for_using}) were obtained  by M.~Timan  (see, for example,  \cite[Ch.~6]{A_Timan_M1960}, \cite[Ch.~2]{M_Timan_M2009}). In the Musielak-Orlicz type spaces, inequalities of the kind as (\ref{S_M.12}) were proved in \cite{Abdullayev_Chaichenko_Shidlich_2021}.



\end{document}